\documentclass[10pt]{amsart}

\usepackage{graphicx, xcolor}
\usepackage{amssymb}
\usepackage{bbm}
\usepackage{amsmath,mathtools,enumerate}
\usepackage{amsthm}
\usepackage{stmaryrd}
\usepackage{tikz-cd}
\usepackage{faktor}

\newcommand{\B}{\mathcal{B}}

\newcommand{\Z}{\mathbb{Z}} 
\newcommand{\N}{\mathbb{N}} 

\renewcommand{\O}{\mathcal{O}}

\newcommand{\E}{\mathcal{E}}
\newcommand{\G}{\mathcal{G}}
\renewcommand{\H}{\mathcal{H}}
\newcommand{\cR}{\mathcal{R}} 
\newcommand{\s}{\mathfrak{s}}

\newcommand{\dad}{\mathrm{dad}}
\newcommand{\asdim}{\mathrm{asdim}}

\theoremstyle{definition}

\newtheorem{thm}{Theorem}[section]
\newtheorem{lem}[thm]{Lemma}
\newtheorem{prop}[thm]{Proposition}
\newtheorem{dfn}[thm]{Definition}
\newtheorem{cor}[thm]{Corollary}
\newtheorem{rmk}[thm]{Remark}

\numberwithin{equation}{section}

\title[Principal Groupoid Models for Cuntz Algebras and their D.A.D.]{Principal Groupoid Models for Cuntz Algebras and their Dynamic Asymptotic Dimension}

\author{Samuel Evington}
\address{Samuel Evington, Mathematical Institute, University of M\"unster, Ein\-stein\-strasse 62, 48149 M\"unster, Germany}
\email{evington@uni-muenster.de}

\author{Philipp Sibbel}
\address{Philipp Sibbel, Mathematical Institute, University of M\"unster, Ein\-stein\-strasse 62, 48149 M\"unster, Germany}
\email{philipp.sibbel@uni-muenster.de}

\dedicatory{Dedicated to the memory of John Evington (1959--2025)}

\thanks{Funded by the Deutsche Forschungsgemeinschaft (DFG, German Research Foundation) under Germany’s Excellence Strategy – EXC 2044 – 390685587, Mathematics Münster – Dynamics – Geometry – Structure; the Deutsche Forschungsgemeinschaft (DFG, German Research Foundation) – Project-ID 427320536 – SFB 1442; and ERC Advanced Grant 834267 - AMAREC}

\begin{document}

\begin{abstract}
We compute the dynamic asymptotic dimension of the principal groupoid models for the Cuntz algebras $\mathcal{O}_k$ for $2 \leq k < \infty$ that have arisen from work of Winter and the authors.
Our method generalises to a wide class of Deaconu--Renault groupoids.
As an application of our results, we prove that $\O_2$ has infinitely many non-conjugate C$^*$-diagonals with Cantor spectrum, and we generalise this result to other Cuntz algebras by combining the main result with work of Kopsacheilis--Winter and Brown--Clark--Sierakowski--Sims.
\end{abstract}

\maketitle

\section*{Introduction}
\numberwithin{equation}{section}
\renewcommand{\thethm}{\Alph{thm}}

Building on the foundational work of Renault (\cite{renault1980groupoid}) and Connes (\cite{Connes79,Connes82}), the study of groupoid C$^*$-algebras has developed into a significant area of modern C$^*$-algebraic research with deep connections to non-commutative geometry (\cite{ConnesBook}), topological dynamics (\cite{renault1980groupoid}), Cartan subalgebras (\cite{kumjian1986c,renault2008cartan}), the universal coefficient theorem (\cite{Tu99, barlak2017cartan}), and the Elliott classification programme (\cite{li2020every,li2022constructing}).

An important early example that demonstrated the potential of groupoid C$^*$-algebras was Renault's étale groupoid model for the Cuntz algebra $\O_k$ (\cite{renault1980groupoid}). For $k \in \N_{\geq 2}$, the Cuntz algebra $\O_k$ is defined to be the universal C$^*$-algebra generated by $k$ isometries $S_1,\ldots,S_k$ such that $\sum_{i=1}^k S_iS_i^* = 1$ (\cite{cuntz1977simple}). Renault's étale groupoid model is given by
\begin{equation*}
    \G = \{(x,m-n,y): x,y \in W, \, m,n \in \N, \,  \sigma^m(x) = \sigma^n(y)\},  
\end{equation*}
where $W = \{1,\ldots,k\}^\N$ is a Cantor space and $\sigma:W \rightarrow W$ is the left shift. The product of $(x_1,\ell_1,y_1), (x_2,\ell_2,y_2) \in \G$ is defined whenever $y_1 = x_2$ and is given by $(x_1,\ell_1+\ell_2, y_2)$ and the dynamics of $\sigma$ defines a natural étale topology on $\G$. This basic construction was latter generalised by Deaconu (\cite{deaconu1995groupoids}), defining the large class of étale groupoids that are known today as Deaconu--Renault groupoids (see Definition \ref{def:DR-groupoids} below).

Renault observed that his groupoid model for $\O_k$ is not \emph{principal}, meaning that there are points $x \in \G^{(0)}$ in the unit space with non-trivial isotropy subgroups $\G_x^x = \{g \in \G: r(g) = s(g) = x\}$. Indeed, $\G_x^x \cong \Z$ for any eventually periodic sequence $x \in W$. 
A consequence of this lack of principality is that the Cartan subalgebra of $\O_k$ arising as $C(\G^{(0)}) \subseteq C^*_r(\G)$ fails to have the unique extension property of pure states and so is not a C$^*$-diagonal in the sense of Kumjian (\cite{kumjian1986c}). In particular, this Cartan subalgebra fails to have finite diagonal dimensional (\cite{LLW23}). Similarly, having isotropy groups isomorphic to $\Z$ forces the dynamic asymptotic dimension of $\G$ to be infinite (\cite{guentner2017dynamic}). For these reasons, Renault's groupoid models for the Cuntz algebras $\O_k$ most likely lie outside the class of those groupoid models that one could reasonably hope to classify by a $K$-theoretic invariant. 

The main object of study in this paper is a family of Deaconu--Renault groupoids introduced in \cite{ES25}, building on \cite{sibbel2024cantor}. These groupoids were used to construct the first \emph{principal} étale groupoid models for the Cuntz algebras $\O_k$ and have a unit space homeomorphic to the Cantor space. They therefore give rise to C$^*$-diagonals in $\O_k$ with Cantor spectrum. The main result of this paper is a computation of the dynamic asymptotic dimension (see Definition \ref{def:DAD} below) of the groupoids in this family, which we use to construct non-conjugate C$^*$-diagonals in $\O_k$ with Cantor spectrum.

The dynamic asymptotic dimension of a topological groupoid was introduced by Guentner--Willett--Yu, motivated by coarse geometrical considerations, and has proven to be useful in the study of non-commutative covering dimension for groupoid C$^*$-algebras (\cite{guentner2017dynamic, courtney2024alexandrov,BonLi24}).  For the case of our principal étale groupoid models for $\O_k$, we find that the dynamic asymptotic dimension is one.

\begin{thm}\label{intro-thm:dad-1}
    Let $k \in \N_{\geq 2}$. There exists a principal, second-countable, locally compact, Hausdorﬀ, étale groupoid $\G_k$ whose  unit space is a Cantor space such that $C^*_r(\G_k) \cong \O_k$ and $\dad(\G_k) = 1$.
\end{thm}
Our method, motivated by the argument of Guentner--Willett--Yu for $\Z$-crossed products, extends to a large class of Deaconu--Renault groupoids (see Theorem~\ref{thm:dad-1} below for the most general result).
By combining Theorem \ref{intro-thm:dad-1} with results of Bönicke (\cite{Bon24}) and Banakh--Banakh (\cite{Banakh22}), we can define infinitely many distinct principal étale groupoid models for $\O_2$ by taking groupoid products and using that $\O_2 \otimes \O_2 \cong \O_2$. By Kumjian--Renault theory, it follows that there are infinitely many non-conjugate C$^*$-diagonals in $\O_2$ with Cantor spectrum.
\begin{thm}\label{intro-thm:O-2}
    There exists a countably infinite family of principal, second-countable, locally compact, Hausdorﬀ, étale groupoid models for $\O_2$ with Cantor unit space that are mutually non-isomorphic. Hence, there are infinitely many non-conjugate C$^*$-diagonals with Cantor spectrum in $\O_2$. 
\end{thm}

The distinguishing invariant for the groupoid models in Theorem \ref{intro-thm:O-2} is the dynamic asymptotic dimension and all values in $\N_{\geq 1} \cup \{\infty\}$ occur. In the process, we also prove that the C$^*$-diagonals in $\O_2$ constructed by Sibbel--Winter have infinite diagonal dimension answering a question from \cite{sibbel2024cantor}.

Since $\O_k$ is not self-absorbing for $2 < k < \infty$, a different argument is needed in this setting. Our approach is based on the tensorial absorption of uniformly hyperfinite (UHF) algebras and depends on the parity of $k$.

When $k$ is even, the Cuntz algebra $\O_k$ tensorially absorbs the CAR algebra  $M_{2^\infty}$. Therefore, we can fully generalise Theorem \ref{intro-thm:O-2} by combining Theorem \ref{intro-thm:dad-1} with the recent construction of a non-canonical C$^*$-diagonal in the CAR algebra $M_{2^\infty}$ by Kopsacheilis--Winter (\cite{kopsacheilis2025folding}). In the case where $k$ is odd, the best we are able to show using current methods is that there are at least two distinct principal étale groupoid models. This result makes use of Theorem \ref{intro-thm:dad-1} combined with the principal étale groupoid models of the Kirchberg algebra $M_{k^\infty} \otimes \O_\infty$ resulting from the construction of Brown--Clark--Sierakowski--Sims (\cite{brown2016purely}).  

\begin{thm}\label{intro-thm:O-k}
Let $k \in \N_{\geq 2}$.
\begin{enumerate}[(i)]
    \item Suppose $k$ is even. Then for every $n \in \N_{\geq 1}\cup \{\infty\}$ there exists a principal, second-countable, locally compact, Hausdorﬀ, étale groupoid $\G_{k,n}$ with Cantor unit space such that $C^*_r(\G_{k,n}) \cong \O_k$ and  $\dad(\G_{k,n}) = n$. Hence, there are infinitely many non-conjugate C$^*$-diagonals in $\O_k$ with Cantor spectrum. 
    \item Suppose $k$ is odd. Then there are two non-isomorphic groupoid models for $\O_k$ that are principal, second-countable, locally compact, Hausdorﬀ, étale and have Cantor unit space. Hence, there are at least two non-conjugate C$^*$-diagonals in $\O_k$ with Cantor spectrum. 
\end{enumerate}    
\end{thm}

A generalisation of the Kopsacheilis--Winter construction from the $M_{2^\infty}$-case to more general UHF algebras would allow a unified approach to the $k$-odd and $k$-even cases of Theorem \ref{intro-thm:O-k}. Alternatively, a proof that the Brown--Clark--Sierakowski--Sims groupoids have finite dynamic asymptotic dimension would allow a strengthening of the statement in the $k$-odd case to infinitely many mutually non-isomorphic groupoid models. 

\subsection*{Acknowledgements}
The authors would like to thank Queen's University Belfast and the Isaac Newton Institute for Mathematical Sciences for their support and hospitality during the programme \emph{Topological Groupoids and their C$^*$-Algebras}, which was funded by EPSRC grant EP/V521929/1.

\section{Preliminaries}
\renewcommand{\thethm}{\arabic{thm}}
\numberwithin{thm}{section}

\subsection{Cartan subalgebras and C$^*$-diagonals}

We briefly recall the definition of Cartan subalgebras (in the sense of Renault \cite{renault2008cartan}) and C$^*$-diagonals (originally due to Kumjian \cite{kumjian1986c}). 

\begin{dfn}[{cf.\ \cite[Definition 5.1]{renault2008cartan}}]
Let $A$ be a C$^*$-algebra. A non-degenerate sub-C$^*$-algebra $D \subseteq A$ is called a \emph{Cartan subalgebra} if 
\begin{enumerate}
\item[(1)] $D$ is a maximal abelian sub-C$^*$-algebra of $A$,
\item[(2)] $A$ is generated by $\mathcal{N}_A(D) = \{n \in A : n^*Dn \subseteq D \text{ and } nDn^* \subseteq D\}$, and
\item[(3)] there exists a faithful conditional expectation of $A \rightarrow D$.
\end{enumerate}
Moreover, $D \subseteq A$ is called a C$^*$-\textit{diagonal} if additionally
\begin{enumerate}
\item[(4)] every pure state on $D$ has a unique extension to a pure state on $A$.
\end{enumerate}
We say two Cartan subalgebras (resp.\ C$^*$-diagonals) $D_1$ and $D_2$ in $A$ are \textit{conjugate} if there exists an automorphism of $A$ sending $D_1$ to $D_2$.
\end{dfn}

\subsection{Groupoids and their C$^*$-algebras}

We refer the reader to \cite{renault1980groupoid, renault2008cartan} for a general introduction to the theory of topological groupoids and the associated C$^*$-algebras. The main purpose of this subsection is to fix the notation and recall the key results of Kumjian--Renault theory.

For a topological groupoid $\G$, we write $\G^{(0)}$ for the unit space and $r,s: \G \to \G^{(0)}$ for the range and source maps.  In this paper, all groupoids considered will be locally compact and Hausdorff, and almost all of them will be second-countable too. However, for clarity, we shall include all topological hypothesis in the formal statements of our results. Moreover, we are primarily interested in \emph{étale} groupoids, meaning that $r$ and $s$ are local homeomorphisms.

For each $x\in \G^{(0)}$, the subgroup
$ \G_x^x = \{g\in \G : r(g)=s(g)=x \}$ is called isotropy subgroup at $x$.
The groupoid $\G$ is \emph{principal} if every unit $x \in \G^{(0)}$ has trivial isotropy, i.e.\ $\G_x^x=\{x\}$; the groupoid $\G$ is \emph{topologically principal} if the set of units with trivial isotropy is dense in $\G^{(0)}$.

The product $\G_1 \times \G_2$ of two topological groupoids $\G_1,\G_2$ is defined via component-wise operations and the product topology. Inductively, one can then define finite products of topological groupoids. 
For topological groupoids $(\G_i)_{i\in\N}$, the infinite product groupoid is the topological groupoid given by 
\begin{equation}
\prod_{i\in\N} (\G_i,\mathcal{G}_i^{(0)}) = \{ (\gamma_i)_{i\in\N} \in \prod_{i\in\N} \G_i : \gamma_i \in \G_i^{(0)} \text{ for almost all } i \in \N\}
\end{equation}
with component-wise operations and the product topology.  Given a topological groupoid $\G$, we write $\G^n$ for the $n$-fold product $\G \times \dots \times \G$ and $\G^\infty$ for the infinite product $\prod_{i\in\N} (\G,\mathcal{G}^{(0)})$.
The following properties of topological groupoids are preserved under finite and infinite products: principality, the étale property, second-countability, and the Hausdorff property. Local compactness is preserved under finite products and passes to the infinite product if the unit spaces are compact.

For the definition of a twist $\Sigma$ over a locally compact, Hausdorff, étale groupoid $\G$ and the construction of the reduced C$^*$-algebra $C_r^*(\G,\Sigma)$, we refer the reader to \cite[Section 4]{renault2008cartan}. The C$^*$-algebra $C_r^*(\G,\Sigma)$ contains a canonical abelian subalgebra isomorphic to $C(G^{(0)})$.
The following theorem summarises the fundamental results of Kumjian--Renault theory; see \cite[Theorems 5.2 and 5.9]{renault2008cartan} and \cite[Proposition~5.11]{renault2008cartan}.

\begin{thm}[{\cite{kumjian1986c,renault2008cartan}}] \label{thm:renault-cartan} There is a one-to-one correspondence between Cartan subalgebras and twisted étale groupoids. More precisely: 
\begin{enumerate}[(i)]
  \item Let $\G$ be a topologically principal, second-countable, locally compact, Hausdorff, étale groupoid and let $\Sigma$ be a twist over $\G$. Then $C_r^*(\G,\Sigma)$ is separable and $C_0(\G^{(0)}) \subseteq C_r^*(\G,\Sigma)$ is a Cartan subalgebra. Moreover, $C_0(\G^{(0)}) \subseteq C_r^*(\G,\Sigma)$ is a C$^*$-diagonal if and only if $\G$ is principal.
   \item Let $A$ be a separable C$^*$-algebra and $D \subseteq A$ be a Cartan subalgebra. Then there is a topologically principal, second-countable, locally compact, Hausdorff, étale groupoid $\G$ and a twist $\Sigma$ such that the inclusion $D \subseteq A$ is isomorphic to $C_0(\G^{(0)}) \subseteq C_r^*(\G,\Sigma)$. Moreover, the twisted groupoid $(\G,\Sigma)$ is uniquely determined up to isomorphism by the inclusion $D \subseteq A$.  
\end{enumerate}
\end{thm}

Finally, we recall the connection between the product of étale groupoids and the minimal tensor product of C$^*$-algebras. For second-countable, locally compact, Hausdorff, étale groupoids $\G_1$ and $\G_2$, we have
\begin{equation}
    C^*_r(\G_1 \times \G_2) = C^*_r(\G_1) \otimes C^*_r(\G_2)
\end{equation}
 (see \cite[Lemma 5.1]{barlak2017cartan}). Similarly, for second-countable, locally compact, Hausdorff, étale groupoids $(\G_i)_{i\in\N}$ each with a compact unit space, we have
\begin{equation}
    C^*_r\left(\prod_{i\in\N} (\G_i,\G_i^{(0)})\right) \cong \bigotimes_{i=1}^\infty C_r^*(\G_i),
\end{equation}
with the minimal tensor product on the right-hand side (see \cite[Lemma 5.2]{barlak2017cartan}).

\subsection{Deaconu--Renault groupoids}

The following construction was introduced in its general form by Deaconu (\cite{deaconu1995groupoids}) building on work of Renault (\cite{renault1980groupoid}). 

\begin{dfn}
\label{def:DR-groupoids}

Let $X$ be a locally compact Hausdorff space and $\sigma$ a local homeomorphism on $X$.
The \textit{Deaconu-Renault groupoid} corresponding to $\sigma$ is given by
\begin{equation}
\G(X,\sigma) = \{(x,m - n,y) : m,n \in \N,x,y\in X \text{ and } \sigma^m(x) = \sigma^n(y)\},
\end{equation}
with groupoid structure operations
$(x,k,y)(w,l,z) = (x,k + l,z)$ whenever $y=w$ and $(x,k,y)^{-1} = (y,-k,x)$, and basic open sets 
\begin{equation}
Z(U,m,n,V) = \{(x,m-n,y) : (x,y) \in U \times V \text{ with } \sigma^m(x) = \sigma^n(y)\}
\end{equation}
where $m,n \in \N$ and $U$,$V$ are open sets in $X$.
\end{dfn}

Deaconu proved that $\G(X,\sigma)$ is a locally compact, Hausdorﬀ, étale groupoid that is second-countable whenever $X$ is second-countable (see \cite{deaconu1995groupoids, renault2000cuntz}).
The unit space $\G(X,\sigma)^{(0)} = \{(x,0,x): x \in X\}$ is canonically homeomorphic to $X$ and it is standard practice to identify these two spaces. 
Furthermore, for any $m,n \in \N$ and any compact subsets $K_1, K_2 \subseteq X$,   
\begin{equation}
    Z(K_1, m, n, K_2) = \{(x,m-n,y) : (x,y) \in K_1 \times K_2 \text{ with } \sigma^m(x) = \sigma^n(y)\}
\end{equation} is compact since  $\{(x,y) \in K_1 \times K_2 : \sigma^m(x) = \sigma^n(y)\}$ is compact and the map $(x,y) \mapsto (x,m-n,y)$ is continuous
(see the proof of \cite[Proposition~3.2]{exel2007semigroups} for example). 

We isolate a couple of technical results about Deaconu--Renault groupoids that we will use in the main body of the paper. 
\begin{lem}\label{lem:S1}
   Let $\G = \G(X, \sigma)$ be the Deaconu--Renault groupoid for a local homeomorphism $\sigma:X \rightarrow X$ on a locally compact Hausdorff space $X$. Suppose $(x_0,\ell, x_1) \in Z(X,m,n,X)$ where $n,m \leq N$. Then $\sigma^{i}(x_0) = \sigma^{i-\ell}(x_1)$ for all $i \geq N$.
\end{lem}
\begin{proof}
    Since $(x_0,\ell, x_1) \in Z(X,m,n,X)$, it follows that $n = m - \ell$ and $\sigma^{m}(x_0) = \sigma^{n}(x_1) = \sigma^{m-\ell}(x_1)$. Let $i \geq N$. Then $i-m \geq 0$ and so we can apply $\sigma^{i-m}$ to both sides and obtain $\sigma^{i}(x_0) = \sigma^{i-\ell}(x_1)$. 
\end{proof}

\begin{prop}\label{prop:S2}
    Let $\G = \G(X, \sigma)$ be the Deaconu--Renault groupoid for a local homeomorphism $\sigma:X \rightarrow X$ on a compact Hausdorff space $X$. 
    Let $K \subseteq \G$ be a relatively compact subset.
    Let $\langle K \rangle$ be the subgroupoid of $\G$ generated by $K$.
    Suppose $\langle K \rangle \subseteq X \times [-R, R] \times X$ for some $R \in \N$. Then $\langle K \rangle$ is relatively compact.
\end{prop}
\begin{proof}
    Without loss of generality, we may assume $K = K^{-1}$.
    We have $\G = \bigcup_{n,m \in \N} Z(X, m, n, X)$ and each set $Z(X, m, n, X)$ is open. Since  $K \subseteq \G$ is relatively compact, $K \subseteq \bigcup_{n,m \leq  N} Z(X, m, n, X)$ for some $N \in \N$.
    Let $g \in \langle K \rangle$. Then $g = g_1g_2 \dots g_k$ for elements $g_1,g_2 \ldots,g_k \in K$ such that the product is well-defined. Say $g_j = (x_{j-1},\ell_j,x_j)$ for some $x_0,\ldots,x_k \in X$ and $l_1, \ldots, l_k \in \Z$. Then $g = (x_0,\ell,x_k)$ where $\ell = \ell_1 + \dots + \ell_k$. Since each $g_j \in Z(X, m_j, n_j, X)$ for some $m_j,n_j \leq N$, we can apply Lemma \ref{lem:S1} repeatedly to obtain 
\begin{equation}
    \begin{split}
        \sigma^i(x_0) &= \sigma^{i-\ell_1}(x_1)\\
        &= \sigma^{i-\ell_1 - \ell_2}(x_2)\\
        &= \dots\\
        &= \sigma^{i-\ell_1 - \ell_2 - \dots - \ell_k}(x_k)       
    \end{split}
    \end{equation}
provided $i \geq \max(N, N+\ell_1, N + (\ell_1+\ell_2), \ldots, N + (\ell_1 + \ell_2 + \dots + \ell_{k-1}))$.

By hypothesis, $\langle K \rangle \subseteq X \times [-R, R] \times X$, so 
\begin{equation}
    \max(N, N+\ell_1, N + (\ell_1+\ell_2), \ldots, N + (\ell_1 + \ell_2 + \dots + \ell_{k-1})) \leq N+R.
\end{equation}

Hence, $\sigma^{N+R}(x_0) =  \sigma^{N+R-\ell}(x_k)$, i.e.\ $g \in Z(X, N+R, N+R-\ell, X)$. Since $|\ell| \leq R$, we have
\begin{equation}
   g \in \bigcup_{n,m \leq N+2R}Z(X, m, n, X).
\end{equation}
Since $g \in \langle K \rangle$ was arbitrary, it follows that 
\begin{equation}\label{eqn:K-is-rel-compact}
  \langle K \rangle \subseteq \bigcup_{n,m \leq N+2R}Z(X, m, n, X).
\end{equation}
However, the right hand side of \eqref{eqn:K-is-rel-compact} is a finite union of compact sets, as each $Z(X, m, n, X)$ is compact. Hence, $\langle K \rangle$ is relatively compact.
\end{proof}

\subsection{Minimal homeomorphisms on compact Hausdorff spaces}

We recall that a homeomorphism $\alpha:X \rightarrow X$ on a topological space is called \emph{minimal} if the only closed, invariant subsets are $\emptyset$ and $X$.
The following facts about minimal homeomorphisms on compact Hausdorff spaces are well-known to experts, but we include short proofs for completeness. 

\begin{lem}\label{lem:U1}
    Let $\alpha:X \rightarrow X$ be a minimal homeomorphism on an infinite compact Hausdorff space $X$.
    For any $N \in \N$, there is a non-empty open set $U \subseteq X$ such that $U \cap \alpha^n(U) = \emptyset$ for $0 < |n| \leq N$.
\end{lem}
\begin{proof}
    Since $X$ is infinite and $\alpha$ is minimal, there are no points with a finite orbit under $\alpha$. Let $N \in \N$. Pick $x \in X$. The points $\alpha^n(x)$ for $|n| \leq N$ are all distinct. Since $X$ is Hausdorff, there are $2N+1$  disjoint open sets $U_{-N},\ldots,U_N$ such that $\alpha^n(x) \in U_n$ for all $|n| \leq N$. Set $U = \bigcap_{n = -N}^N \alpha^{-n}(U_n)$. 
    Then $\alpha^{n}(U) \subseteq U_n$ for all $|n| \leq N$. Hence, $U \cap \alpha^n(U) = \emptyset$ for $0 < |n| \leq N$ because $U_{-N},\ldots,U_N$ are disjoint.
\end{proof}

\begin{lem}\label{lem:minimalHomeoBasicFact}
    Let $\alpha:X \rightarrow X$ be a minimal homeomorphism on a compact Hausdorff space $X$ and let $U \subseteq X$ be open and non-empty. 
    Then there exists $M \in \N$ such that, for all $x \in X$, there exist $m_1, m_2 \in (0,M] \cap \Z$ with $\alpha^{-m_1}(x) \in U$ and $\alpha^{m_2}(x) \in U$. 
\end{lem}
\begin{proof}
    Let $x \in X$. By minimality, the backward and forward half-orbits of $x$ under $\alpha$ are dense. Hence, there are non-zero $m_1^x, m_2^x \in \N$  such that $\alpha^{-m_1^x}(x) \in U$ and $\alpha^{m_2^x}(x) \in U$. As $\alpha$ is continuous, there exists an open neighbourhood $U_x$ of $x$ such that
    $\alpha^{-m_1^x}(U_x) \subseteq U$ and $\alpha^{m_2^x}(U_x) \subseteq U$. 
    Since $X$ is compact, the open cover $\{U_x:x \in X\}$ has a finite subcover $\{U_{x_1},\ldots,U_{x_n}\}$. Set $M = \max \{m_j^{x_i}: 1 \leq i \leq n, 1 \leq j \leq 2\}$. Then, for all $x \in X$, there exist $m_1, m_2 \in (0,M] \cap \Z$  such that $\alpha^{-m_1}(x) \in U$ and $\alpha^{m_2}(x) \in U$.  
\end{proof}

\subsection{Dynamic asymptotic dimension}

Dynamic asymptotic dimension was in introduced by Guentner, Willett, and Yu in \cite{guentner2017dynamic}, motivated by work of Gromov \cite{gromov1993asymptotic} and Roe \cite{roe2003lectures}.
It yields upper bounds on the nuclear dimension of the groupoid C$^*$-algebra and is closely connected to diagonal dimension (see \cite{guentner2017dynamic,courtney2024alexandrov,BonLi24,LLW23}). 
We begin by recalling the definition from \cite[Definition 5.1]{guentner2017dynamic}.

\begin{dfn} 
\label{def:DAD}
 Let $\G$ be a locally compact, Hausdorff, étale groupoid. Then $\G$ has \emph{dynamic asymptotic dimension} $d \in \N$ if $d$ is the smallest natural number with the following property:  For every open relatively compact subset $K$ of $\G$, there are open subsets $U_0,\dots,U_d \subseteq \G^{(0)}$ that cover $s(K) \cup r(K)$ such that, for each $i \in \{0,\dots,d\}$, 
 $\{g \in K : s(g), r(g) \in U_i\}$ is contained in a relatively compact subgroupoid of $\G$.
 If no such $d$ exists then the dynamic asymptotic dimension is infinite.
\end{dfn}

We denote the dynamic asymptotic dimension of a locally compact, Hausdorff, étale groupoid $\G$ by $\dad(\G)$.
The following proposition records some key properties of dynamic asymptotic dimension that will be used in this paper.

\begin{prop} 
\label{prop:dad-finite-product}
Let $\G, \H$ be locally compact, Hausdorff, étale groupoids. 
\begin{enumerate}[(i)]
    \item If $\H$ is a closed subgroupoid of $\G$, then $\dad(\H) \leq \dad(\G)$.
    \item $\dad(\G \times \H) \geq \max(\dad(\G),\dad(\H))$.
    \item $\dad(\G \times \H) \leq \dad(\G) + \dad(\H)$.
    \item Let $(\G_i)_{i \in \N}$ be a sequence of locally compact, Hausdorff, étale groupoids with compact unit spaces. Let $n \in \N$. Then $\dad\left(\prod_{i\in\N} (\G_i,\G_i^{(0)})\right) \geq \dad(\prod_{i=1}^n \G_i)$.
\end{enumerate}
\end{prop}
\begin{proof} ~
    \begin{enumerate}[(i)]
    \item This is \cite[Lemma 3.11]{courtney2024alexandrov}.
    \item  Fix a unit $u \in \H^{(0)}$. Then $\G \times \{u\}$ is a closed étale subgroupoid of $\G \times \H$ that is isomorphic to $\G$. Hence, $\dad(\G) = \dad(\G \times \{u\}) \leq \dad(\G \times \H)$ by (i). Similarly, we have $\dad(\H) \leq \dad(\G \times \H)$.
    \item This is \cite[Theorem A(3)]{Bon24}. 
    \item For each $i > n$, choose a unit $u_i \in \G_i^{(0)}$. Then $\prod_{i=1}^n \G_i \times \prod_{i=n+1}^\infty \{u_i\}$  is a closed subgroupoid of $\prod_{i\in\N} (\G_i,\G_i^{(0)})$ and is isomorphic to $\prod_{i=1}^n \G_i$. The result now follows by (i).\qedhere
    \end{enumerate}
\end{proof}

\subsection{Coarse spaces and their asymptotic dimension}
\label{subsection:asdim}

Our main reference for coarse geometry is \cite{roe2003lectures}.
We recall that a coarse structure $\E$ on a set $X$ is a collection of subsets of $X\times X$ (called \textit{controlled sets} or \textit{entourages}) that contains the diagonal $\Delta_X$ and is stable under subsets, inverses, finite unions, and compositions. The pair $(X,\mathcal{E})$ is called a \emph{coarse space}.

Given two coarse spaces $(X_1,\mathcal{E}_1)$ and $(X_2,\mathcal{E}_2)$, the coarse structure on the product $X_1 \times X_2$ is given by
\begin{equation} \label{BasisProductCoarseStructure}
    \E_1 \ast \E_2 = \{E \subseteq (X_1 \times X_2) \times (X_1 \times X_2): (p_i \times p_i)(E) \in \E_i \mbox{ for } i=1,2\}
\end{equation}
where $p_i : X_1 \times X_2 \to X_i$ denotes the projection to the $i$-th coordinate. This construction extends inductively to finite products. 

Let $\G$ be a locally compact, Hausdorff, étale groupoid. The canonical coarse structure $\E_\G$ on $\G$ consists of the subsets $E\subseteq \G \times \G$ such that 
\begin{equation}
    E \subseteq \{(g,h): r(g) = r(h), g^{-1}h \in K\} \cup \Delta_\G
\end{equation} for some relatively compact open subset $K\subseteq \G$  (see for example \cite{Bon24}).
The following proposition shows that the product of groupoids with compact unit spaces is compatible with the product of coarse spaces defined above. 
\begin{prop} \label{prop:groupoid-product-coarse-structure}
Let $\G_1$ and $\G_2$ be locally compact, Hausdorff, étale groupoids with compact unit spaces.
Then $\E_{\G_1 \times \G_2} = \mathcal{E}_{\G_1}\ast \mathcal{E}_{\G_2}.$
\end{prop}
\begin{proof}    
    We write $\G = \G_1 \times  \G_2$.
    Let $E \in \E_{\G_1 \times \G_2}$. Then there exists an open, relatively compact set $K \subseteq \G_1\times \G_2$ with 
    $E \subseteq \{(g,h) \in \G \times \G:r(g) = r(h), g^{-1}h \in K \} \cup \Delta_\G$. 
    Then for $i \in \{1,2\}$, we have $(p_i \times p_i)(E) \in \mathcal{E}_{\G_i} $ using $K_i = p_i(K)$ as the open, relatively compact subset.

    Now let $E \in \E_{\G_1}\ast \E_{\G_2}$. Then there exist open, relatively compact subsets $K_i \subseteq \G_i$ such that $(p_i \times p_i)(E) \subseteq \{(g,h): g^{-1}h \in K_i \} \cup \Delta_{\G_i}$ for $i \in \{1,2\}$. As $\G_i^{(0)}$ is compact, we can assume that $K_i$ contains $\G_i^{(0)}$ for $i \in \{1,2\}$.
    Now, since $K= K_1 \times K_2 \subseteq \G_1 \times \G_2$ is open and relatively compact, we see that $E \in \mathcal{E}_{\G_1\times\G_2}$.
\end{proof}

Let $(X,\mathcal{E})$ be a coarse space. If $E \in \mathcal{E}$, then a family $\mathcal{U}=\{U_i\}_{i\in I}$ of subsets of $X$ is called \textit{$E$-separated}, if $(U_i \times U_j) \cap E= \emptyset$ for all $i \neq j$, and $\mathcal{U}$ is called \textit{$E$-bounded} if $U_i \times U_i \subseteq E$ for all $i \in I$. We now recall the definition of asymptotic dimension.

\begin{dfn}
A coarse space $(X,\mathcal{E})$ has \textit{asymptotic dimension $d$} if $d$ is the smallest natural number with the following property: For any $E \in \E$, there is $F \in \E$ and a cover $\mathcal{U}=\{U_i\}_{i\in I}$ of $X$ such that $\mathcal{U}$ is $F$-bounded and admits a decomposition $\mathcal{U}= \mathcal{U}_0\sqcup \dots \sqcup \mathcal{U}_d$ where each $\mathcal{U}_i$ is $E$-separated. If no such $d$ exists then the asymptotic dimension is infinite.
\end{dfn}

 Asymptotic dimension was originally introduced by Gromov for metric spaces in \cite{gromov1993asymptotic} and later generalized by Roe to general coarse spaces in \cite[Definition 9.4]{roe2003lectures}. 
The definition above is taken from {\cite[Definition 6.3]{guentner2017dynamic}}; see \cite[Proposition~2.3]{banakh2014asymptotic} and \cite{grave2006asymptotic} for a discussion on equivalent definitions of asymptotic dimension.
The asymptotic dimension of a coarse space $(X,\mathcal{E})$ will be denoted by $\asdim(X,\mathcal{E})$.

The following result of Bönicke connects dynamic asymptotic dimension of an étale groupoid to the asymptotic dimension of the induced coarse space.
\begin{thm}(\cite[Theorem B]{Bon24}) \label{thm:dad-equals-asdim} 
Let $\G$ be a principal, $\sigma$-compact\footnote{Note that $\sigma$-compactness follows from the other conditions in the second-countable case.}, locally compact, Hausdorff, étale groupoid with a zero-dimensional, compact unit space. Let $\E_\G$ be the canonical coarse structure on $\G$. Then $\dad(\G) \geq \asdim(\G, \E_\G)$ with equality whenever $\dad(\G)$ is finite.
\end{thm}

Next, we recall the following result of Banakh--Banakh, which provides a lower bound for the asymptotic dimension of a product of coarse spaces. 
\begin{thm}[{\cite[Theorem 1]{Banakh22}}] \label{thm:lower-bound-asdim}
    Let $(X_1,\mathcal{E}_1),\dots,(X_n,\mathcal{E}_n)$ be coarse spaces with  $\asdim(X_i,\mathcal{E}_i) \geq 1$ for $i=1,\dots,n$. Then
        $\asdim(X_1 \times \dots \times X_n, \mathcal{E}_1\ast\dots \ast \mathcal{E}_n) \geq n$.
\end{thm}

\begin{rmk}
   In the main body of the paper, we shall use Theorems \ref{thm:dad-equals-asdim} and \ref{thm:lower-bound-asdim} to compute a lower bound for the dynamic asymptotic dimension of a product groupoid $\G_1 \times \dots \times \G_n$. In concrete cases, the use of Theorem \ref{thm:lower-bound-asdim} can often be replaced by the observation that if $\N$ coarsely embeds into each $\G_i$ then $\N^n$ embeds coarsely into $\G_1 \times \dots \times \G_n$.
\end{rmk}

\section{Deaconu--Renault groupoids and their dynamic asymptotic dimension}\label{section:dad-DR-groupoid}

In this section, we compute the dynamic asymptotic dimension of any Deaconu--Renault groupoid where the local homeomorphism can be written as the product of a minimal homeomorphism on an infinite compact Hausdorff space and a surjective local homeomorphism on a compact Hausdorff space. This class of groupoids is of particular interest as they have played a key role in the recent construction of principal groupoid models for Kirchberg algebras (and especially Cuntz algebras); see \cite{sibbel2024cantor} and \cite{ES25}.

The following omnibus lemma collects together the elementary facts on these Deaconu--Renault groupoids that we will use in the main result.    
\begin{lem}\label{lem:collection}
    Let $\alpha:X \rightarrow X$ be a minimal homeomorphism on an infinite compact Hausdorff space $X$, and let $\rho: W \rightarrow W$ be a surjective local homeomorphism on a compact Hausdorff space $W$. Let $\G = \G(W \times X,\sigma)$ be the Deaconu--Renault groupoid for the local homeomorphism $\sigma = \rho \times \alpha: W \times X \rightarrow W \times X$. Let $U \subseteq X$.
    \begin{enumerate}[(i)]
        \item For all $\ell \in \Z$, we have $\sigma^\ell(W \times U) = W \times \alpha^\ell(U)$.\footnote{Here, $\sigma^{\ell}(W \times U)$ denotes the preimage of $W \times U$ under $\sigma^{|\ell|}$ when $\ell < 0$.}
        \item If $(z_1,\ell,z_2) \in \G$ and $z_1 \in W \times U$, then $z_2 \in W \times \alpha^\ell(U)$. \label{lem:G1}
        \item Suppose $U \neq \emptyset$ is open. There exists $M \in \N$ such that for all $z_1 \in W \times X$, there exist $m_1, m_2 \in (0,M] \cap \Z$ and $z_0, z_2 \in W \times U$ with $(z_0,m_1,z_1) \in \G$ and $(z_1,m_2,z_2) \in \G$. \label{lem:U2}
    \end{enumerate}
\end{lem}
\begin{proof} ~
    \begin{enumerate}[(i)] 
        \item Immediate as $\rho$ is surjective and $\alpha$ is bijective.
        \item Let $(z_1,\ell,z_2) \in \G$. There exist $m,n \in \N$ such that $\ell= m - n$ and $\sigma^{m}(z_1) = \sigma^{n}(z_2)$. Suppose $z_1 \in W \times U$. Then, by (i), $\sigma^{m}(z_1) \in W \times \alpha^{m}(U)$ and $z_2 \in \sigma^{-n}(\sigma^{m}(x)) \in  W \times \alpha^{m-n}(U) = W \times \alpha^\ell(U)$. 
        \item  By Lemma \ref{lem:minimalHomeoBasicFact}, there exists $M \in \N$ such that for all $x \in X$, there exist $m_1, m_2 \in (0,M] \cap \Z$  with $\alpha^{-m_1}(x) \in U$ and $\alpha^{m_2}(x) \in U$.
        Let $z_1 = (w_1,x_1) \in W \times X$. Choose $m_1,m_2 \in (0,M] \cap \Z$ such that $x_0 = \alpha^{-m_1}(x_1) \in U$ and $x_2 = \alpha^{m_2}(x_1) \in U$. Pick $w_0 \in \rho^{-m_1}(w_1)$ and set $w_2 =\rho^{m_2}(w_1)$. Set $z_0 = (w_0, x_0)$ and $z_2 = (w_2, x_2)$. Then $z_0,z_2 \in W \times U$ and  $(z_0,m_1,z_1), (z_1,m_2,z_2) \in \G$.  \qedhere
    \end{enumerate}    
\end{proof}

We are now ready for the main argument, which can be viewed as a groupoid version of \cite[Theorem 3.1]{guentner2017dynamic}.
\begin{thm}\label{thm:dad-1}
    Let $\alpha:X \rightarrow X$ be a minimal homeomorphism on an infinite compact Hausdorff space $X$ and let $\rho: W \rightarrow W$ be a surjective local homeomorphism on a compact Hausdorff space $W$.
    Let $\G = \G(W \times X,\rho \times \alpha)$ denote the Deaconu--Renault groupoid of the product $\rho \times \alpha:  W \times X \rightarrow W \times X$.
    Then $\dad(\G) = 1$.
\end{thm}

\begin{proof}
    We first show that $\dad(\G) \neq 0$. Let $C = \{(z,1,(\rho \times \alpha)(z)): z \in W \times X\}$. Then $C$ is a compact open subset of $\G$, homeomorphic to $W \times X$ via $r$. Moreover, $C$ generates $\G$ as a groupoid.  Suppose for a contradiction that $\dad(\G) = 0$. Then, by definition, there exists a relatively compact, open subgroupoid of $\G$ containing $C$. 
    Since $C$ generates $\G$, it follows that $\G$ is compact. However, the map $c:\G \rightarrow \Z$ given by $(x,\ell,y) \mapsto \ell$ is continuos and surjective, contradicting the compactness of $\G$.
    Therefore, $\dad(\G) \neq 0$, 
    and it suffices to show that $\mathrm{dad}(\G) \leq 1$.
    
    Let $K \subseteq \mathcal{G}$ be a relatively compact and open subset as in Definition \ref{def:DAD}. Without loss of generality, we may assume that $K=K^{-1}$. 
    As $K$ is relatively compact, there exists $N \in \N$ such that $K \subseteq (W \times X) \times [-N,N] \times (W \times X)$. 
    Since $\alpha$ is minimal and $X$ is an infinite compact Hausdorff space, there is a non-empty open subset $U \subseteq X$ such that 
    \begin{equation}\label{eqn:disjoint-images-U}
        U \cap \alpha^n ( U) = \emptyset
    \end{equation}
    for all $n \in \Z$ with $0 < |n| \leq 5N$ by Lemma \ref{lem:U1}. 
    
    Since $X$ is a compact Hausdorff space, there exists a non-empty open subset $V \subseteq X$ such that $\overline{V} \subseteq U$.
    We now define 
    \begin{align}
        U_0 & = \bigsqcup_{n=-N}^N \alpha^n(U), \label{eqn:def-U0}\\ 
        U_1 &= X \setminus \bigsqcup_{n=-N}^N \alpha^n(\overline{V})\label{eqn:def-U1}.
    \end{align}
    By construction, $\{W \times U_0,W \times U_1\}$ is an open cover of $W \times X$, so is also an open cover of $r(K) \cup s(K)$. Set
    \begin{equation}\label{eqn:def-Ki}
        K_i = \{g \in K : s(g),r(g) \in W \times U_i \}
    \end{equation}
    for $i\in\{0,1\}$. To prove that $\dad(\G) \leq 1$, it suffices to show that the subgroupoid $\langle K_i \rangle$ generated by $K_i$ is relatively compact for $i \in \{0,1\}$.
        
    We start with $K_0$.  
    We claim that $\langle K_0 \rangle$ is contained in $(W \times X) \times [-3N,3N] \times (W \times X)$.
    Assume to the contrary that this is not the case. 
    Then, noting that $K_0 = K_0^{-1}$, there exists a sequence of composible elements $g_1,\ldots,g_m \in K_0$ such that the product satisfies
    \begin{equation}\label{eqn:ell-more-than-3N}
        g_1 \cdots g_m \not\in (W \times X) \times [-3N,3N] \times (W \times X).
    \end{equation} 
    Since each $g_i \in K_0$, we may write $g_i = (y_{i-1},\ell_i,y_i)$ for some $y_0,y_1,\ldots,y_m \in W \times U_0$ and $\ell_1,\ldots,\ell_m \in \Z \cap [-N,N]$. Then  $| \ell_1 + \dots + \ell_m | > 3N$ by \eqref{eqn:ell-more-than-3N}.
    Since each $\ell_i \in [-N,N] \cap \Z$ and $| \ell_1 + \dots + \ell_m | > 3N$, there must exist some $r < m$ such that 
    \begin{equation}\label{eqn:def-r}
        2N < | \ell_1 + \dots + \ell_{r} | \leq 3N.
    \end{equation}

    Since $y_0, y_r \in W \times U_0$,  there are $n_0,n_{r} \in [-N, N] \cap \Z$ such that $y_0 \in W \times \alpha^{n_0}(U)$ and $y_r \in W \times \alpha^{n_r}(U)$ by \eqref{eqn:def-U0}. As $n_0,n_{r} \in [-N, N] \cap \Z$, by the triangle inequality and \eqref{eqn:def-r}, we have
    \begin{equation}\label{eqn:estimate-sum-ells}
        0 < | \ell_1 + \dots + \ell_{r}  +  n_0 - n_r |  \leq 5N.
    \end{equation}
    Since $g_1\cdots g_r = (y_{0}, \ell_1 + \dots + \ell_r, y_r)$ and $y_{0} \in W \times \alpha^{n_0}(U)$, Lemma~\ref{lem:collection}(\ref{lem:G1}) implies that $y_r \in W \times \alpha^{\ell_1 + \dots + \ell_r + n_0}(U)$. Hence, $\alpha^{n_{r}}(U) \cap \alpha^{\ell_1 + \dots + \ell_r + n_0}(U) \neq \emptyset$. Since $\alpha$ is invertible, we have
    \begin{equation}\label{eqn:contradictionU}
        U \cap \alpha^{\ell_1 + \dots + \ell_r + n_0 - n_r}(U) \neq \emptyset.
    \end{equation}
    However, \eqref{eqn:contradictionU} combined with \eqref{eqn:estimate-sum-ells} contradicts the definition of $U$ in \eqref{eqn:disjoint-images-U}.
    Hence, $\langle K_0 \rangle$ is contained in $(W \times X) \times [-3N,3N] \times (W \times X)$ as claimed. Since $K_0 \subseteq K$ and $K$ is relatively compact, $K_0$ is relatively compact. By Proposition~\ref{prop:S2}, we deduce that $\langle K_0 \rangle$ is relatively compact.

    We now consider $K_1$.
    By Lemma \ref{lem:collection}(\ref{lem:U2}), there exists $M \in \N$ such that, for any $z_1 \in W \times X$, there exists $m_1, m_2 \in (0,M] \cap \Z$ and $z_0, z_2 \in W \times V$ with $(z_0,m_1,z_1) \in \G$ and $(z_1,m_2,z_2) \in \G$.
    We claim that the subgroupoid $\langle K_1 \rangle$ generated by $K_1$ is contained in $(W \times X) \times [-(M+N),M+N] \times (W \times X)$.
    Assume to the contrary that this is not the case. 
    Then there exists a sequence of composible elements  $g_1,\ldots,g_m \in K_1$ such that $g_1 \cdots g_m \not\in (W \times X) \times [-(M+N),M+N] \times (W \times X)$, noting that $K_1 = K_1^{-1}$. Since each $g_i \in K_1$, we may write $g_i = (y_{i-1},\ell_i,y_i)$ where $y_0, \ldots,y_m \in W \times U_1$ and $\ell_1,\ldots,\ell_m \in [-N,N] \cap \Z$ with 
    $| \ell_1 + \dots + \ell_m | > M+N$.
    
    We consider first the case $ l_1 + \dots + l_m  > M+N$. By Lemma \ref{lem:collection}(\ref{lem:U2}), we may chose $m_2 \in (0,M] \cap \Z$ and $z_2 \in W \times V$ such that $(y_0, m_2, z_2) \in \G$.
    Since each $\ell_i \in [-N,N] \cap \Z$, there must exist some $r < m$ such that $m_2 \leq \ell_1 + \dots + \ell_{r}  \leq N+ m_2$. Hence, 
    \begin{equation}\label{eqn:def-r-m1}
         0 \leq \ell_1 + \dots + \ell_{r} - m_2 \leq N.
    \end{equation}
    Since $(y_0, m_2, z_2)^{-1}g_1\cdots g_r = (z_2, \ell_1 + \dots + \ell_r - m_2, y_r) \in \G$ and $z_2 \in W \times V$, it follows from Lemma \ref{lem:collection}(\ref{lem:G1}) that $y_r \in W \times \alpha^{\ell_1 + \dots + \ell_r - m_2}(V)$, but then $y_r \not\in W \times U_1$  by \eqref{eqn:def-r-m1} and \eqref{eqn:def-U1}, which contradicts the fact that $g_r \in K_1$. 
    
    Next, we consider the case $\ell_1 + \dots + \ell_m  < -(M+N)$. By Lemma \ref{lem:collection}(\ref{lem:U2}), we may choose $m_1 \in (0,M] \cap \Z$ and $z_0 \in W \times V$ such that $(z_0, m_1, y_0) \in \G$.
    Since each $\ell_i \in [-N,N] \cap \Z$, there must exist some $r < m$ with $-m_1 \geq \ell_1 + \dots + \ell_{r}  \geq -(N + m_1)$. Hence,
    \begin{equation}\label{eqn:def-r-m2}
         0 \geq \ell_1 + \dots + \ell_{r} + m_1 \geq -N.
    \end{equation}
    Since $(z_0, m_1, y_0)g_1\cdots g_r = (z_0, \ell_1 + \dots + \ell_r + m_1, y_r) \in \G$ and $z_0 \in W \times V$, it follows from Lemma \ref{lem:collection}(\ref{lem:G1}) that $y_r \in W \times \alpha^{\ell_1 + \dots + \ell_r + m_1}(V)$, but then $y_r \not\in W \times U_1$  by \eqref{eqn:def-r-m2} and \eqref{eqn:def-U1}, which contradicts the fact that $g_r \in K_1$.  

    As both cases lead to a contradiction, we conclude that $\langle K_1 \rangle$ is contained in $(W \times X) \times [-(M+N),M+N] \times (W \times X)$. Since $K_1 \subseteq K$ and $K$ is relatively compact, $K_1$ is relatively compact. By Proposition~\ref{prop:S2}, we deduce that $\langle K_1 \rangle$ is relatively compact. This completes the proof that $\dad(\G) \leq 1$. 
\end{proof}

We now deduce Theorem \ref{intro-thm:dad-1} from the introduction.
\begin{proof}[Proof of Theorem \ref{intro-thm:dad-1}]
    Let $k \in \N_{\geq 2}$ and $\alpha:X \rightarrow X$ be a minimal homeomorphism on a Cantor space.
    Let $\G_{k,\alpha}$ be the Deaconu--Renault groupoid of the local homeomorphism $\rho \times \alpha$ on $\{1,\dots,k\}^\N \times X$, where $\rho$ is the left shift on $\{1,\dots,k\}^\N$, and let $\mathcal{R}_\s$ be the canonical étale equivalence relation for the UHF algebra $M_\s$ with supernatural number $\s$ (see for example \cite[Theorem 3.1.15]{renault1980groupoid} or \cite{GPS04}). 
    
    In \cite[Theorem B]{ES25}, it was shown that, for specific choices of $\s$ and $\alpha$, the product groupoid $\G_k = \G_{k,\alpha} \times \mathcal{R}_\s$ satisfies $C^*_r(\G_k) \cong \O_k$; see also \cite[Corollary 5.4]{ES25}. Moreover, it was shown that $\G_k$ is principal, second-countable, locally compact, Hausdorff and étale with a Cantor unit space.    
    
    By Theorem \ref{thm:dad-1}, $\dad(\G_{k,\alpha}) = 1$. Moreover, $\dad(\mathcal{R}_\s) = 0$ since $\mathcal{R}_\s$ is a locally finite groupoid (see for example \cite[Example 5.3]{guentner2017dynamic}).
    By Proposition \ref{prop:dad-finite-product}(ii) and Proposition \ref{prop:dad-finite-product}(iii), we have
    \begin{equation}
       1 = \dad(\G_{k,\alpha}) \leq  \dad(\G_{k,\alpha} \times \mathcal{R}_\s) \leq \dad(\G_{k,\alpha}) + \dad(\mathcal{R}_\s) = 1.
    \end{equation} 
    Hence, $\dad(\G_k) = 1$.
\end{proof}

\begin{rmk}
In \cite{hjelmborgstability}, Hjelmborg showed that the Cuntz algebra $\mathcal{O}_2$ can be constructed as the C$^*$-algebra of a principal Deaconu--Renault groupoid where the local homeomorphism is given as the product of the left shift on $\{1,2\}^\N$ and an irrational rotation on the circle $\mathbb{T}$. By Theorem \ref{thm:dad-1}, this groupoid also has dynamic asymptotic dimension one.
\end{rmk}

\section{Applications}

In this section, we shall combine Theorem \ref{intro-thm:dad-1}, proved in the previous section, with some results about the behaviour of dynamic asymptotic dimension under taking products to prove Theorems \ref{intro-thm:O-2} and \ref{intro-thm:O-k} from the introduction.

We begin with the case of finite products. 
\begin{thm} \label{thm:lower-bound-dad}
    Let $\G_1, \ldots, \G_n$ be principal, second-countable, locally compact, Hausdorff, étale groupoids with compact and zero-dimensional unit spaces. Suppose that $\dad(\G_i) \geq 1$ for every $i \in \{1,\dots,n\}$. Then
    \begin{equation}\label{eqn:lower-bound-dad}
        \dad(\G_1 \times \dots \times \G_n) \geq n,
    \end{equation}
    and equality holds whenever $\dad(\G_i) = 1$ for every $i \in \{1,\dots,n\}$.
\end{thm}
\begin{proof}
    Suppose first $\dad(\G_{i_0}) = \infty$ for some $i_0 \in \{1,\dots,n\}$. Then, by Proposition \ref{prop:dad-finite-product}(ii),  $\dad(\G_1 \times \dots \times \G_n) \geq \dad(\G_{i_0}) = \infty$ and \eqref{eqn:lower-bound-dad} follows.
    
    Now suppose that $\dad(\G_{i}) < \infty $ for all $i \in \{1,\dots,n\}$. By Theorem \ref{thm:dad-equals-asdim}, 
    $\asdim(\G_i, \mathcal{E}_{\G_i}) = \dad(\G_i) $, where $\mathcal{E}_{\G_i}$ denotes the canonical coarse structure on $\G_i$.
    By Proposition~\ref{prop:dad-finite-product}(iii),  $\dad(\G_1 \times \dots \times \G_n) \leq \dad(\G_1) + \cdots + \dad(\G_n) < \infty$.
    By Theorem \ref{thm:dad-equals-asdim}, Proposition \ref{prop:groupoid-product-coarse-structure} and Theorem \ref{thm:lower-bound-asdim}, we compute that
    \begin{equation} \label{eqn:dad-asdim-1}
    \begin{split}
        \dad(\G_1 \times \dots \times \G_n) &= \asdim(\G_1 \times \dots \times \G_n, \mathcal{E}_{\G_1 \times \dots \times \G_n})\\
        &= \asdim(\G_1 \times \dots \times \G_n, \mathcal{E}_{\G_1} \ast \dots \ast \mathcal{E}_{\G_n}) \\
        &\geq n.
    \end{split}
    \end{equation}
    If $\dad(\G_i) = 1$ for every $i \in \{1,\dots,n\}$, Proposition \ref{prop:dad-finite-product}(iii) gives the upper bound $\dad(\G_1 \times \dots \times \G_n) \le n,$
    implying equality in \eqref{eqn:lower-bound-dad}.
\end{proof}

Next, we deduce an analogous result for infinite products. 
\begin{cor}\label{cor:dad-infinity}
    Let $(\G_i)_{i \in \N}$ be a sequence of principal, second-countable, locally compact, Hausdorff, étale groupoids each with a zero-dimensional, compact unit space. Suppose that $\dad(\G_i) \geq 1$ for all $i \in \N$. Then
    \begin{equation}
        \dad\left(\prod_{i\in\N} (\G_i,\G_i^{(0)})\right) = \infty.
    \end{equation}  
\end{cor}
\begin{proof}
    Let $n \in \N$. 
    Then, by Proposition \ref{prop:dad-finite-product}(iv) and Theorem \ref{thm:lower-bound-dad},
    \begin{equation}
        \dad\left(\prod_{i\in\N} (\G_i,\G_i^{(0)})\right) \ge \dad(\G_1 \times \dots \times \G_n) \geq n.
    \end{equation}
    Since $n$ was arbitrary, the result follows.
\end{proof}

We are now able to prove Theorem \ref{intro-thm:O-2} from the introduction.
\begin{proof}[Proof of Theorem \ref{intro-thm:O-2}] 
    Let $\G_2$ be the principal, second-countable, locally compact, Hausdorﬀ, étale groupoid with $C_r^*(\G_2) = \mathcal{O}_2$ and $\dad(\G_2) = 1$ that is given by Theorem \ref{intro-thm:dad-1}.
    
    For each $n \in \N$, the $n$-fold product groupoid $\G_2^n$ is also principal, second-countable, locally compact, Hausdorﬀ, étale and has a Cantor unit space. Since $\O_2 \otimes \O_2 \cong \O_2$ by \cite{RordamElliott}, we have
    \begin{equation}
    C^*_r(\G_2^n) = C^*_r(\G_2)^{\otimes n} = \mathcal{O}_2^{\otimes n} \cong \mathcal{O}_2.
    \end{equation}
    Hence, we have found countably many principal groupoid models for $\mathcal{O}_2$ with Cantor spectrum. By Theorem \ref{thm:lower-bound-dad}, we have $\dad(\G_2^n)= n$. Therefore, the groupoids $\G_2^n$ are mutually non-isomorphic as $n$ varies.
    By Theorem \ref{thm:renault-cartan}, it then follows that there are infinitely many non-conjugate C$^*$-diagonals in $\mathcal{O}_2$ with Cantor spectrum.
\end{proof}

\begin{rmk}
    In \cite{sibbel2024cantor}, the first examples of C$^*$-diagonals with Cantor spectrum in $\O_2$ were constructed. Translated into the language of groupoids, the construction of \cite{sibbel2024cantor} is as follows: Let $\G_{2,\alpha} = \G(\{1,2\}^\N \times X, \rho \times \alpha)$ be the Deaconu--Renault groupoid where $\rho:\{1,2\}^\N \rightarrow \{1,2\}^\N$ is the left shift and $\alpha:X \rightarrow X$ is any minimal homeomorphism on a Cantor space $X$. Then the main argument of \cite{sibbel2024cantor} shows that 
    \begin{equation}
        C^*_r((\mathcal{G}_{2,\alpha})^\infty)\cong C^*_r(\mathcal{G}_{2,\alpha})^{\otimes \infty} \cong \mathcal{O}_2.
    \end{equation}
      By Theorem \ref{thm:dad-1}, $\dad(\G_{2,\alpha}) = 1$. By Corollary \ref{cor:dad-infinity}, $\dad((\mathcal{G}_{2,\alpha})^\infty) = \infty$.

      In light of \cite[Theorem E]{LLW23}, it follows that diagonal dimension of the associated C$^*$-diagonal in $\mathcal{O}_2$ is infinite, which answers the second part of \cite[Question (ii)]{sibbel2024cantor}. Furthermore, it follows that the C$^*$-diagonals in $\mathcal{O}_2$ constructed in \cite{sibbel2024cantor} and in \cite{ES25} are not conjugate, answering a question posed in \cite[Remark 5.5]{ES25}.
\end{rmk}

We now turn to the case of the Cuntz algebras $\O_k$ where $k \in \N_{\ge2}$. In this case, we shall use the following UHF-stability result in place of $\O_2 \cong \O_2 \otimes \O_2$. 

\begin{lem}\label{lem:UHF-stable}
    Let $k,n \in \N_{\geq 2}$ with $\mathrm{gcd}(k-1,n) = 1$. Then $\mathcal{O}_k \otimes M_{n^\infty} \cong \mathcal{O}_k$.
\end{lem}
\begin{proof}
    By the Künneth formula (\cite{schochet1982topological}), $K_1(\O_k \otimes M_{n^\infty}) = 0$ and we have 
    \begin{equation}
    K_0(\O_k \otimes M_{n^\infty}) \cong \Z_{k-1} \otimes \Z[n^{-1}] \cong \Z_{k-1},
    \end{equation}
    where the second isomorphism is given by multiplication using the fact that $n$ is invertible modulo $k-1$. Moreover, the unit of $\O_k \otimes M_{n^\infty}$ maps under these isomorphisms to $1 \in \Z_{k-1}$. As $\O_k$ and $\O_k \otimes M_{n^\infty}$ are UCT Kirchberg algebras, the result now follows by the Kirchberg--Phillips theorem (\cite{Ki95,Ph00}).
\end{proof}

In the case where $k$ is even, Lemma \ref{lem:UHF-stable} gives $\O_k \otimes M_{2^\infty} \cong \O_k$, and we can take advantage of the following recent result of Kopsacheilis--Winter.
\begin{thm}(\cite[Theorem A]{kopsacheilis2025folding})
\label{thm:Kopsacheilis-Winter}
    Let $G = \Z \rtimes \Z_2$ be the infinite dihedral group and let $X$ be a Cantor space.
    There exists a free minimal action of $\phi:G \curvearrowright X$ such that 
    \begin{equation}\label{eqn:Kopsacheilis-Winter}
         (C(X) \rtimes_\phi G) \otimes M_{2^\infty} \cong M_{2^\infty}, 
    \end{equation}
    but the C$^*$-diagonal $C(X) \otimes D_{2^\infty} \subseteq (C(X) \rtimes_\phi G) \otimes M_{2^\infty}$ is not conjugate to the canonical C$^*$-diagonal $D_{2^\infty} \subseteq M_{2^\infty}$.  
\end{thm}

For a group $G$ acting on a space $X$, we denote the transformation groupoid by $ X \rtimes_\phi G$. Note that $(C(X) \rtimes_\phi G) \otimes M_{2^\infty}$, the left hand side of \eqref{eqn:Kopsacheilis-Winter}, is isomorphic to the C$^*$-algebra of the groupoid $(X \rtimes_\phi G) \times \cR_{2^\infty}$, where $\cR_{2^\infty}$ is the canonical groupoid model for the UHF algebra $M_{2^\infty}$. We now compute the dynamic asymptotic dimension of this groupoid. 

\begin{prop}\label{prop:dad-of-KW-1}
    Let $G = \Z \rtimes \Z_2$ be the infinite dihedral group, $X$ be a Cantor space, and $\phi:G \curvearrowright X$ be a free minimal action.
    Then $\dad((X \rtimes_\phi G) \times \cR_{2^\infty})  = 1$.
\end{prop}
\begin{proof}
  By \cite[Theorem D]{BonLi24}, $\dad(X \rtimes_\phi G) = 1$. 
  Moreover, $\dad(\mathcal{R}_{2^\infty}) = 0$ since $\mathcal{R}_{2^\infty}$ is a locally finite groupoid (see for example \cite[Example 5.3]{guentner2017dynamic}).
  The result now follows by Proposition \ref{prop:dad-finite-product}(ii) and Proposition \ref{prop:dad-finite-product}(iii).
\end{proof}

We now prove the $k$-even part of of Theorem \ref{intro-thm:O-k}.
\begin{thm}\label{thm:models-for-Oeven}
    Let $k \in \N_{\geq 2}$ be even. For each $n \in \N_{\geq 1} \cup \{\infty\}$, there exists a principal, second-countable, locally compact, Hausdorff, étale groupoid $\G_{k,n}$ with a Cantor unit space such that $C_r^*(\G_{k,n}) \cong \O_k$ and $\dad(\G_{k,n}) = n$. 
\end{thm}
\begin{proof}
    Let $k \in \N_{\geq 2}$ be even. Let $\G_{k,1}$ be the principal groupoid such that $C_r^*(\G_{k,1}) \cong \O_k$ and $\dad(\G_{k,1}) = 1$ from Theorem \ref{intro-thm:dad-1}.
    Let $G = \Z \rtimes \Z_2$ be the infinite dihedral group, $X$ be a Cantor space, and $\phi:G \curvearrowright X$ be the free minimal action from Theorem \ref{thm:Kopsacheilis-Winter}.
    Let $\H = (X \rtimes_\phi G) \times \cR_{2^\infty}$. By Theorem \ref{thm:Kopsacheilis-Winter}, $C_r^*(\H) \cong M_{2^\infty}$.
    By Proposition \ref{prop:dad-of-KW-1}, we have $\dad(\H) = 1$.
    
    For $n \in \N_{\geq 2}$, set $\G_{k,n} = \G_{k,1} \times \H^{n-1}$ and set $\G_{k,\infty} = \G_{k,1} \times \H^\infty$. 
    Since $\G_{k,1}$ and $\H$ are both principal, second-countable, locally compact, Hausdorff, étale groupoids with Cantor unit spaces, the same is true of $\G_{k,n}$ for all $n \in \N \cup \{\infty\}$.
    
    By Theorem \ref{thm:lower-bound-dad}, we have  $\dad(\G_{k,n}) = n$ for all finite $n$, and by Corollary \ref{cor:dad-infinity}, we have $\dad(\G_{k,\infty}) = \infty$. 
    Since $k$ is even, $\O_k \otimes M_{2^\infty} \cong \O_k$ by Lemma \ref{lem:UHF-stable}. Using this and the fact that $M_{2^\infty}$ is strongly self-absorbing, we have 
    \begin{equation}
        C_r^*(\G_{k,n}) \cong C_r^*(\G_{k,1}) \otimes C_r^*(\H)^{\otimes n-1} \cong \O_k \otimes M_{2^\infty}^{\otimes n-1} \cong \O_k \otimes M_{2^\infty} \cong \O_k
    \end{equation} 
    for all $n > 1$, and similarly, $C_r^*(\G_{k,\infty}) \cong \O_k$.
\end{proof}

By Lemma \ref{lem:UHF-stable}, we have  $\mathcal{O}_k \otimes M_{k^\infty} \cong \mathcal{O}_k$ for all $k \in \N_{\geq 2}$.
Therefore, if one could generalise the Kopsacheilis--Winter result to cover the UHF algebras $M_{k^\infty}$ for $k \neq 2$, the above argument would generalise easily to the $\O_k$-setting. However, it is not known whether this is possible, as the construction of Kopsacheilis and Winter is quite specific to the CAR algebra setting. 
In this paper, we peruse an alternative approach using the Kirchberg algebra $M_{k^\infty} \otimes \O_\infty$ and the principal groupoid models constructed by Brown--Clark--Sierakowski--Sims in \cite{brown2016purely}.

\begin{prop}\label{prop:BCSS}
 Let $k \in \N_{\geq 2}$. There exists a principal, locally compact, Hausdorff, étale groupoid $\B_k$ with Cantor unit space such that $C_r^*(\B_k) \cong M_{k^\infty} \otimes \O_\infty$.
\end{prop}
\begin{proof}
    As $K_0(M_{k^\infty} \otimes \O_\infty) \cong K_0(M_{k^\infty}) \cong \Z[k^{-1}]$, it can be given the structure of a dimension group where $[1_{M_{k^\infty}} \otimes 1_{\O_\infty}]$ is positive. Moreover, $K_1(M_{k^\infty} \otimes \O_\infty) = 0$. Hence, \cite[Theorem 6.1]{brown2016purely} applies and there is a principal, ample, étale groupoid $\B_k$ with a compact unit space such that $C_r^*(\B_k) \cong M_{k^\infty} \otimes \O_\infty$. From the construction, it is clear that $\B_k$ is locally compact, Hausdorff and second-countable; see \cite[Proof of Theorem 6.1]{brown2016purely}. To see that the unit space of $\B_k$ is a Cantor space, we observe that one can use a refinement of the standard Brattelli diagram for $M_{k^\infty}$ in \cite[Lemma~6.2]{brown2016purely}, so the corresponding graph groupoid has a Cantor unit space. 
\end{proof}

Using the Brown--Clark--Sierakowski--Sims groupoids, we can construct two non-conjugate C$^*$-diagonals in $\O_k$ for all $k \in \N_{\geq 2}$. Note that this argument doesn't require $k$ to be odd, it is merely the case that we can prove a stronger result using the Kopsacheilis--Winter approach when $k$ is even.

\begin{thm}\label{thm:BCSS}
Let $k \in \N_{\geq 2}$. There exist principal, locally compact, Hausdorff, étale groupoids $\G_k$ and $\B_k$ both having a Cantor unit space such that $C_r^*(\G_{k}) \cong C_r^*(\G_{k} \times \B_k) \cong  \O_k$ but $\G_{k} \not\cong \G_{k} \times \B_k$.
\end{thm}
\begin{proof}
    Let $\G_{k}$ be the groupoid with $C_r^*(\G_{k}) \cong \O_k$ and $\dad(\G_k) = 1$ from Theorem~\ref{intro-thm:dad-1}.
    Let $\B_k$ be the groupoid with $C_r^*(\B_k) \cong M_{k^\infty} \otimes \O_\infty$ from Proposition~\ref{prop:BCSS}. Then both groupoids are principal, locally compact, Hausdorff, étale and have a Cantor unit space.  

    By Lemma \ref{lem:UHF-stable} and the $\mathcal{O}_\infty$-stability of Kirchberg algebras, we have
    \begin{equation}
        \O_k \otimes (M_{k^\infty} \otimes \O_\infty) \cong (\O_k \otimes M_{k^\infty}) \otimes \O_\infty \cong  \O_k \otimes \O_\infty \cong \O_k \label{eqn:Ok-Mk-Oinfty}.
    \end{equation}
    It follows that  $C_r^*(\G_{k} \times \B_k) \cong \O_k \cong C_r^*(\G_{k})$.
    Since $C_r^*(\B_k)$ is not an AF algebra, we must have $\dad(\B_k) \geq 1$; see for example \cite[Theorem 8.6]{guentner2017dynamic}.
    Thus, by Theorem~\ref{thm:lower-bound-dad}, we have $\dad(\G_{k} \times \B_k) \geq 2$. Since $\dad(\G_{k}) = 1 \neq \dad(\G_{k} \times \B_k)$, the two groupoids are not isomorphic. 
\end{proof}

We can now formally prove Theorem \ref{intro-thm:O-k} from the introduction.
\begin{proof}[Proof of Theorem \ref{intro-thm:O-k}]
    Part (i) follows from Theorem \ref{thm:models-for-Oeven} and Theorem \ref{thm:renault-cartan}. Part (ii) follows from Theorem \ref{thm:BCSS} and Theorem \ref{thm:renault-cartan}. 
\end{proof}

We conclude this paper with a couple of remarks.

\begin{rmk}
    If the Brown--Clark--Sierakowski--Sims groupoids $\B_k$ were to satisfy $\dad(\B_k) < \infty$, then it would follow that $\dad(\G_k \times B_k^{n})$ attains infinitely many distinct values as $n$ varies, and we would have infinitely many distinct principal étale groupoid models for $\O_k$ for all $k \in \N_{\geq 2}$. 
\end{rmk}
\begin{rmk}
    In light of \cite[Remark 6.8]{LLW23}, we expect that the dynamic asymptotic dimension of a principal étale groupoid $\G$ with Cantor unit space is the same as the diagonal dimension of $C(\G^{(0)}) \subseteq C_r^*(\G)$. If this were true, we could also state our main theorems with diagonal dimension as the distinguishing invariant instead of the dynamic asymptotic dimension.
\end{rmk}

\bibliographystyle{abbrv}

\begin{thebibliography}{10}

\bibitem{Banakh22}
I.~Banakh and T.~Banakh.
\newblock On the asymptotic dimension of products of coarse spaces.
\newblock {\em Topology Appl.}, 311, Paper No. 107953, 8 pp., 2022.

\bibitem{banakh2014asymptotic}
T.~Banakh, O.~Chervak, and N.~Lyaskovska.
\newblock Asymptotic dimension and small subsets in locally compact topological
  groups.
\newblock {\em Geom. Dedicata}, 169:383--396, 2014.

\bibitem{barlak2017cartan}
S.~Barlak and X.~Li.
\newblock Cartan subalgebras and the {UCT} problem.
\newblock {\em Adv. Math.}, 316:748--769, 2017.

\bibitem{Bon24}
C.~B\"onicke.
\newblock On the dynamic asymptotic dimension of \'etale groupoids.
\newblock {\em Math. Z.}, 307(16), Paper No. 16, 16 pp., 2024.

\bibitem{BonLi24}
C.~B\"onicke and K.~Li.
\newblock Nuclear dimension of subhomogeneous twisted groupoid
  \textnormal{C}$^*$-algebras and dynamic asymptotic dimension.
\newblock {\em Int. Math. Res. Not. IMRN}, 2024(16):11597--11610, 2024.

\bibitem{brown2016purely}
J.~H. Brown, L.~O. Clark, A.~Sierakowski, and A.~Sims.
\newblock Purely infinite simple \textnormal{C}$^*$-algebras that are principal
  groupoid \textnormal{C}$^*$-algebras.
\newblock {\em J. Math. Anal. Appl.}, 439(1):213--234, 2016.

\bibitem{Connes79}
A.~Connes.
\newblock Sur la th{\'e}orie non commutative de l'int{\'e}gration.
\newblock In {\em Alg\`ebres d'op\'erateurs ({S}\'em., {L}es {P}lans-sur-{B}ex,
  1978)}, volume 725 of {\em Lecture Notes in Math.}, pages 19--143. Springer,
  Berlin, 1979.

\bibitem{Connes82}
A.~Connes.
\newblock A survey of foliations and operator algebras.
\newblock In {\em Operator algebras and applications, {P}art 1 ({K}ingston,
  {O}nt., 1980)}, volume~38 of {\em Proc. Sympos. Pure Math.}, pages 521--628.
  Amer. Math. Soc., Providence, RI, 1982.

\bibitem{ConnesBook}
A.~Connes.
\newblock {\em Noncommutative geometry}.
\newblock Academic Press, Inc., San Diego, CA, 1994.

\bibitem{courtney2024alexandrov}
K.~Courtney, A.~Duwenig, M.~C. Georgescu, A.~an~Huef, and M.~G. Viola.
\newblock Alexandrov groupoids and the nuclear dimension of twisted groupoid
  \textnormal{C}$^*$-algebras.
\newblock {\em J. Funct. Anal.}, 286(9), Paper No. 110372, 49 pp., 2024.

\bibitem{cuntz1977simple}
J.~Cuntz.
\newblock Simple \textnormal{C}$^*$-algebras generated by isometries.
\newblock {\em Comm. Math. Phys}, 57(2):173--185, 1977.

\bibitem{deaconu1995groupoids}
V.~Deaconu.
\newblock Groupoids associated with endomorphisms.
\newblock {\em Trans. Amer. Math. Soc.}, 347(5):1779--1786, 1995.

\bibitem{ES25}
S.~Evington and P.~Sibbel.
\newblock \textnormal{C}$^*$-diagonals with {Cantor} spectrum in {Cuntz}
  algebras.
\newblock arXiv:2506.22163.

\bibitem{exel2007semigroups}
R.~Exel and J.~Renault.
\newblock Semigroups of local homeomorphisms and interaction groups.
\newblock {\em Ergodic Theory and Dynam. Systems}, 27(6):1737--1771, 2007.

\bibitem{GPS04}
T.~Giordano, I.~Putnam, and C.~Skau.
\newblock Affable equivalence relations and orbit structure of {C}antor
  dynamical systems.
\newblock {\em Ergodic Theory Dynam. Systems}, 24(2):441--475, 2004.

\bibitem{grave2006asymptotic}
B.~Grave.
\newblock Asymptotic dimension of coarse spaces.
\newblock {\em New York J. Math.}, 12:249--256, 2006.

\bibitem{gromov1993asymptotic}
M.~Gromov.
\newblock Asymptotic invariants of infinite groups.
\newblock In {\em Geometric group theory, {V}ol.\ 2 ({S}ussex, 1991)}, volume
  182 of {\em London Math. Soc. Lecture Note Ser.}, pages 1--295. Cambridge
  University Press, Cambridge, 1993.

\bibitem{guentner2017dynamic}
E.~Guentner, R.~Willett, and G.~Yu.
\newblock Dynamic asymptotic dimension: relation to dynamics, topology, coarse
  geometry, and \textnormal{C}$^*$-algebras.
\newblock {\em Math. Ann.}, 367(1-2):785--829, 2017.

\bibitem{hjelmborgstability}
J.~v.~B. Hjelmborg.
\newblock {\em On Stability and Pure Infiniteness of
  \textnormal{C}$^*$-Algebras}.
\newblock PhD thesis, University of Southern Denmark, 2000.
\newblock Available at https://web.math.ku.dk/rordam/students/jh-thesis.ps.

\bibitem{Ki95}
E.~Kirchberg.
\newblock Exact \textnormal{C}$^*$-algebras, tensor products, and the
  classification of purely infinite algebras.
\newblock In {\em Proceedings of the {I}nternational {C}ongress of
  {M}athematicians, {V}ol.\ 1, 2 ({Z}\"urich, 1994)}, pages 943--954.
  Birkh\"auser, Basel, 1995.

\bibitem{kopsacheilis2025folding}
G.~Kopsacheilis and W.~Winter.
\newblock Paper-folding models for the {CAR} algebra.
\newblock arXiv:2508.04837.

\bibitem{kumjian1986c}
A.~Kumjian.
\newblock On \textnormal{C}$^*$-diagonals.
\newblock {\em Canad. J. Math.}, 38(4):969--1008, 1986.

\bibitem{LLW23}
K.~Li, H.~Liao, and W.~Winter.
\newblock The diagonal dimension of sub-\textnormal{C}$^*$-algebras.
\newblock arXiv:2303.16762.

\bibitem{li2020every}
X.~Li.
\newblock Every classifiable simple \textnormal{C}$^*$-algebra has a {Cartan}
  subalgebra.
\newblock {\em Invent. Math.}, 219(2):653--699, 2020.

\bibitem{li2022constructing}
X.~Li.
\newblock Constructing {Menger} manifold \textnormal{C}$^*$-diagonals in
  classifiable \textnormal{C}$^*$-algebras.
\newblock {\em Int. Math. Res. Not. IMRN}, 2022(23):18992--19053, 2022.

\bibitem{Ph00}
N.~C. Phillips.
\newblock A classification theorem for nuclear purely infinite simple
  \textnormal{C}$^*$-algebras.
\newblock {\em Doc. Math.}, 5:49--114, 2000.

\bibitem{renault1980groupoid}
J.~Renault.
\newblock {\em A groupoid approach to \textnormal{C}$^*$-algebras}, volume 793
  of {\em Lecture Notes in Math.}
\newblock Springer, Berlin, 1980.

\bibitem{renault2000cuntz}
J.~Renault.
\newblock Cuntz-like algebras.
\newblock In {\em Operator theoretical methods ({T}imi\c soara, 1998)}, pages
  371--386. The Theta Foundation, Bucharest, 2000.

\bibitem{renault2008cartan}
J.~Renault.
\newblock Cartan subalgebras in \textnormal{C}$^*$-algebras.
\newblock {\em Irish Math. Soc. Bull.}, 61:29--63, 2008.

\bibitem{roe2003lectures}
J.~Roe.
\newblock {\em Lectures on coarse geometry}, volume~31 of {\em Univ. Lecture
  Ser.}
\newblock American Mathematical Society, Providence, RI, 2003.

\bibitem{RordamElliott}
M.~R{\o}rdam.
\newblock A short proof of {E}lliott's theorem: {$\mathcal{O}_2\otimes
  \mathcal{O}_2\cong \mathcal{O}_2$}.
\newblock {\em C. R. Math. Rep. Acad. Sci. Canada}, 16(1):31--36, 1994.

\bibitem{schochet1982topological}
C.~Schochet.
\newblock Topological methods for \textnormal{C}$^*$-algebras. {II}.
  {Geometric} resolutions and the {K{\"u}nneth} formula.
\newblock {\em Pacific J. Math.}, 98(2):443--458, 1982.

\bibitem{sibbel2024cantor}
P.~Sibbel and W.~Winter.
\newblock A {C}antor spectrum diagonal in $\mathcal{O}_2$.
\newblock {\em Proc. Amer. Math. Soc. Ser. B}.
\newblock to appear, arXiv:2409.03511.

\bibitem{Tu99}
J.-L. Tu.
\newblock La conjecture de {B}aum-{C}onnes pour les feuilletages moyennables.
\newblock {\em $K$-Theory}, 17(3):215--264, 1999.

\end{thebibliography}

\end{document}